\documentclass[11pt]{article}
\usepackage{latexsym}
\usepackage{mathrsfs}
\usepackage{indentfirst}
\usepackage{amsmath,amsthm, amssymb}
\usepackage{graphicx}

\usepackage{color}

\usepackage{comment}

\newtheorem{theorem}{Theorem}[section]
\newtheorem{lemma}[theorem]{Lemma}
\newtheorem{corollary}[theorem]{Corollary}

\newtheorem{claim}{Claim}
\newtheorem{observation}{Observation}

\begin{document}
\title{Directed Hamilton cycles in digraphs and matching alternating Hamilton cycles in bipartite graphs
\thanks{Research supported by National Natural Science Foundation of China (10801077 and 11201158), Natural Science Foundation of Guangdong Province, China (9451030007003340) and the Natural Science
Foundation of the Jiangsu Higher Education Institutions of China (08KJB110008).}}
\author{Zan-Bo Zhang$^1$, Xiaoyan Zhang$^{2,3}$\thanks{Corresponding author. Email address:
royxyzhang@gmail.com.}, Xuelian Wen$^4$
\\ \small $^1$Department of Computer Engineering, Guangdong Industry Technical College, \\ \small Guangzhou 510300,
China
\\ \small $^2$School of Mathematical Science and Institute of Mathematics, Nanjing Normal University, \\ \small Nanjing 210023, China
\\ \small $^3$Faculty of Electrical Engineering, Mathematics and Computer Science, University of Twente,
\\ \small P.O. Box 217, 7500 AE Enschede, The Netherlands
\\ \small $^4$School of Economics and Management, South China Normal University,
\\ \small Higher education mega center, Panyu, Guangzhou 510006, China}

\date{}
\maketitle
\begin{abstract}
In 1972, Woodall raised the following Ore type condition for
directed Hamilton cycles in digraphs: Let $D$ be a digraph. If for
every vertex pair $u$ and $v$, where there is no arc from $u$ to
$v$, we have $d^+(u)+d^-(v)\geq |D|$, then $D$ has a directed
Hamilton cycle. By a correspondence between bipartite graphs and
digraphs, the above result is equivalent to the following result
of Las Vergnas:
Let $G = (B,W)$ be a balanced bipartite graph. If for any $b \in
B$ and $w \in W$, where $b$ and $w$ are nonadjacent, we have $d(w)
+d(b) \geq |G|/2 + 1$, then every perfect matching of $G$ is
contained in a Hamilton cycle.

The lower bounds in both results are tight. In this paper, we
reduce both bounds by $1$, and prove that the conclusions still
hold, with only a few exceptional cases that can be clearly
characterized.
$\newline$ $\newline$\noindent\textbf{Key words}: degree sum,
 Matching alternating Hamilton cycle, Hamilton cycle.
\end{abstract}
\section{Introduction}
Hamiltonian problems, and their many variations, have been studied
extensively for more than half a century.
The readers could refer to the surveys of Gould (\cite{Gould1991}
and \cite{Gould2003}), Kawarabayashi (\cite{Kaw2001}) and Broersma
(\cite{Broersma2002}) to trace the development in this field.
Recently, approximate solutions of many traditional Hamiltonian
problems and conjectures in digraphs came forth (\cite{KKO2008},
\cite{KKO2009}, \cite{CKKO2010} and \cite{KOT2010}), which are
surveyed by K\"{u}hn and Osthus (\cite{KuhOstpreprint}).

Hamiltonicity and related properties are also important in
practical applications. For example, in network design, the
existence of Hamilton cycles in the underlying topology of an
interconnection network provide advantage for the routing
algorithm to make use of a ring structure, while the existence of
a hamiltonian decomposition allows the load to be equally
distributed, making network robust (\cite{BDDP1998}). 

There are lots of degree or degree sum conditions for
hamiltonicity. Often, the lower bounds in such conditions are best
possible. However, we could still reduce the bounds and try to
identify all exceptional graphs, that is, the extremal graphs for
the conditions. Such kind of research often leads to the discovery
of interesting topology structures. In this paper, we apply this
idea to Woodall's condition for the existence of directed Hamilton
cycles in digraphs.
\section{Terminology, notations and preliminary results}

In this paper we consider finite, simple and connected graphs, and
finite and simple digraphs. For the terminology not defined in
this paper, the reader is referred to \cite{BM1976} and
\cite{BG2001}.

Let $G$ be a graph with vertex set $V(G)$ and edge set $E(G)$. We
denote by $\nu$ or $|G|$ the order of $V(G)$. For $u\in V(G)$, we
denote by $d(u)$ the degree of $u$, and $N(u)$ or $N_G(u)$ the set
of neighbors of $u$ in $G$. For a subgraph $H$ of $G$ and a vertex
$u\in V(G-H)$, we also denote by $N_{H}(u)$ the set of
neighbors of $u$ in $H$. For any two disjoint vertex sets $X$, $Y$
of $G$ we denote by $e(X,Y)$ the number of edges of $G$ from $X$
to $Y$. For $u, v\in V(G)$, we denote by $d(u,v)$ the distance
between $u$ and $v$, that is, the length of the shortest path
connecting $u$ and $v$. By $uv+$ ($uv-$) we mean the vertices $u$
and $v$ are adjacent (nonadjacent). If a vertex $u$ sends (no)
edges to $X$, where $X$ is a subgraph or a vertex subset of $G$,
we write $u\rightarrow X$ ($u\nrightarrow X$). By $nK_2$, we
denote a graph consisting of $n$ independent edges.

Let $D$ be a digraph with vertex set $V(D)$ and arc set $A(D)$,
$u$, $v$ and $w$ distinct vertices of $D$. We denote by $|D|$ the
order of $V(D)$, $d^+(u)$ and $d^-(u)$ the out-degree and
in-degree of $u$, respectively. The degree of $u$ is the sum of
its out-degree and in-degree. The minimum out-degree and in-degree
of the vertices in $D$, is denoted by $\delta^+(D)$ and
$\delta^-(D)$. We let $\delta^0(D)=min\{\delta^+(D),\
\delta^-(D)\}$. Let $(u,v)$ denote an arc from $u$ to $v$. If
$(u,v)\in A(D)$ or $(v,u)\in A(D)$, we say that $u$ and $v$ are
adjacent. If $(w,u)\in A(D)$ and $(w,v)\in A(D)$, then we say that
the pair $\{u,v\}$ is dominated, if $(u,w)\in A(D)$ and $(v,w)\in
A(D)$, then we say that the pair $\{u,v\}$ is dominating. The
complete digraph on $n\geq 1$ vertices, denoted by
$\overleftrightarrow{K}_n$, is obtained from the complete graph
$K_n$ by replacing every edge $xy$ with two arcs $(x, y)$ and $(y,
x)$. Without causing ambiguity, we use $I_n$ to denote a graph or
a digraph consisting of $n$ independent vertices. A
\emph{transitive tournament} is an orientation of complete graph
for which the vertices can be numbered in such a way that $(i,j)$
is an edge if and only if $i<j$.

Let $C=u_{0}u_{1}\ldots u_{m-1}u_{0}$ be a cycle in a graph $G$.
Throughout this paper, the subscript of $u_{i}$ is reduced modulo
$m$. We always orient $C$ such that $u_{i+1}$ is the successor of
$u_{i}$.
For $0\leq i,j\leq m-1$, the path $u_{i}u_{i+1}\ldots u_{j}$ is
denoted by $u_{i}C^{+}u_{j}$, while the path $u_{i}u_{i-1}\ldots
u_{j}$ is denoted by $u_{i}C^{-}u_{j}$. For a path
$P=v_{0}v_{1}\ldots v_{p-1}$ and $0\leq i,j\leq p-1$, the segment
of $P$ from $v_{i}$ to $v_{j}$ is denoted by $v_{i}Pv_{j}$.

A \emph{matching} $M$ of $G$ is a subset of $E(G)$ in which no two
elements are adjacent. If every $v\in V(G)$ is covered by an edge
in $M$ then $M$ is said to be a \emph{perfect matching} of $G$.
For a matching $M$, an \emph{$M$-alternating path}
(\emph{$M$-alternating cycle}) is a path (cycle) of which the
edges appear alternately in $M$ and $E(G)\backslash M$. We call an
edge in $M$ or an $M$-alternating path starting and ending with
edges in $M$ a \emph{closed $M$-alternating path}, while an edge
in $E(G)\backslash M$ or an $M$-alternating path starting and
ending with edges in $E(G)\backslash M$ an \emph{open
$M$-alternating path}. $\newline$

The following results of Dirac and Ore for the existence of
Hamilton cycles in graphs are basic and famous.
\begin{theorem} (Dirac, 1952 \cite{Dirac1952})
If $G$ is a simple graph with $|G|\geq 3$ and every vertex of $G$
has degree at least $|G|/2$, then $G$ has a Hamilton cycle.
\end{theorem}

\begin{theorem} (Ore, 1960 \cite{Ore1960})
Let $G$ be a simple graph. If for every distinct nonadjacent vertices
$u$, $v$ of $G$, we have $d(u) + d(v) \geq |G|$, then $G$ has a
Hamilton cycle.
\end{theorem}

Below are some of their digraph versions.

\begin{theorem} \label{theorem:Ghouila-Houri} (Ghouila-Houri, 1960
\cite{Ghouila1960}) Let $D$ be a strong digraph. If the degree of
every vertex of $D$ is at least $|D|$, then $D$ has a directed
Hamilton cycle.
\end{theorem}

\begin{theorem} \label{theorem:BanGut2001} (\cite{BG2001}, Corollary 5.6.3)
If $D$ is a digraph with $\delta^0(D)\geq |D|/2$, then $D$ has a
directed Hamilton cycle.
\end{theorem}

\begin{theorem} \label{theorem:Woodall} (Woodall, 1972 \cite{Woodall1972})
Let $D$ be a digraph. If for every vertex pair $u$ and $v$, where
there is no arc from $u$ to $v$, we have $d^+(u)+d^-(v)\geq |D|$,
then $D$ has a directed Hamilton cycle.
\end{theorem}

It is not hard to verify  that the bounds in above theorems are
tight. Nash-Williams \cite{Nash-Williams1969} raised the problem
of describing all the extremal digraphs in Theorem
\ref{theorem:Ghouila-Houri}, that is, all digraphs with minimum
degree at least $|D|-1$, who do not have a directed Hamilton
cycle. As a partial solution to this problem, Thomassen proved a
structural theorem on the extremal graphs.
\begin{theorem} \label{theorem:Thomassen} (Thomassen, 1981 \cite{Thomassen1981})
Let $D$ be a strong non-Hamiltonian digraph, with minimum degree
$|D|-1$. Let $C$ be a longest directed cycle in $D$. Then any two
vertices of $D-C$ are adjacent, every vertex of $D-C$ has degree
$|D|-1$ (in $D$), and every component of $D-C$ is complete.
Furthermore, if $D$ is strongly $2$-connected, then $C$ can be
chosen such that $D-C$ is a transitive tournament.
\end{theorem}

Darbinyan characterized the digraphs of even order that are
extremal for both Theorem \ref{theorem:Ghouila-Houri} and Theorem
\ref{theorem:BanGut2001}.
\begin{theorem} \label{theorem:Darbinyan} (Darbinyan, 1986 \cite{Darbinyan1986})
Let $D$ be a digraph of even order such that the degree of every
vertex of $D$ is at least $|D|-1$ and $\delta^0(D)\geq |D|/2-1$.
Then either $D$ is hamiltonian or $D$ belongs to a non-empty
finite family of non-hamiltonian digraphs.
\end{theorem}

We study the extremal graphs of Theorem \ref{theorem:Woodall} in
this paper. Compared with Theorem \ref{theorem:Thomassen} and
Theorem \ref{theorem:Darbinyan}, we can completely determine all
the extremal graphs.

For other results on degree sum conditions for the existence of
Hamilton cycles in digraphs see \cite{BGH1996}, \cite{BGL1996},
\cite{BGY1999}, \cite{Darbinyan1986}, \cite{Darbinyan1998},
\cite{Manoussakis1992}, \cite{Meyniel1973}, \cite{ZM1991},
\cite{ZZ1999}, and a good summary in chapter 5 of \cite{BG2001}.
$\newline$

Another interesting aspect of directed Hamilton cycle problems is
their connection with the problem of matching alternating Hamilton
cycles in bipartite graphs. Given a bipartite graph $G$ with a
perfect matching $M$, if we orient the edges of $G$ towards the
same part, then contracting all edges in $M$, we get a digraph
$D$. An $M$-alternating Hamilton cycle of $G$ corresponds to a
directed Hamilton cycle of $D$, and vice versa. Hence, Theorem
\ref{theorem:Woodall} is equivalent to the following theorem.

\begin{theorem} \label{theorem:Las}
(Las Vergnas, 1972 \cite{Las1972}) Let $G = (B,W)$ be a balanced
bipartite graph of order $\nu$. If for any $b \in B$ and $w \in
W$, where $b$ and $w$ are nonadjacent, we have $d(w) + d(b) \geq
\nu/2 + 2$, then for every perfect matching $M$ of $G$, there is
an $M$-alternating Hamilton cycle.
\end{theorem}

Hence, we also determine the extremal graphs for the result of Las
Vergnas in this paper.

Theorem \ref{theorem:Las} is an instance of the problem of cycles
containing matchings, which studies the conditions that enforce
certain matchings to be contained in certain cycles. Some related
works can be found in \cite{AFG2009}, \cite{AFGW2007},
\cite{Berman1983}, \cite{Haggkvist1979}, \cite{JW1990},
\cite{Kaw2002}, \cite{Philpotts2008} and \cite{Wojda1983}. In
particular, Berman proved the following.

\begin{theorem} \label{theorem:Berman}
(Berman, 1983 \cite{Berman1983}) Let $G$ be a graph on $\nu \geq
3$ vertices. If for any pair of independent vertices $x$, $y \in V
(G)$, we have $d(x) + d(y) \geq \nu + 1$, then every matching lies
in a cycle.
\end{theorem}

Similarly  to the above-mentioned works, Jackson and Wormald
determined all the extremal graphs of a generalized version of
Berman's result.
\begin{theorem}
(Jackson and Wormald, 1990 \cite{JW1990}) Let $G$ be a graph on
$\nu$ vertices and $M$ be a matching of $G$ such that (1) $d(x) +
d(y) \geq \nu$ for all pairs of independent vertices $x, y$ that
are incident with $M$. Then $M$ is contained in a cycle of $G$
unless equality holds in (1) and several exceptional cases happen.
\end{theorem}

We will state our main results and their proofs in the following
sections.

\section{Main results}

Let $m,\ n\geq 1$ be integers. Let $\mathcal{D}_1$ be the set of
all digraphs obtained by identifying one vertex of
$\overleftrightarrow{K}_{n+1}$ with one vertex of
$\overleftrightarrow{K}_{m+1}$. Let $D_2$ be an arbitrary digraph
on $n$ vertices, and take a copy of $I_{n+1}$. Let $\mathcal{D}_2$
be the set of all digraphs obtained by adding arcs of two
directions between every vertex of $I_{n+1}$ and every vertex of
$D_2$. Let $D_3$ be as shown in Figure \ref{figure:exception3_D},
and take a copy of $\overleftrightarrow{K}_n$. Let $\mathcal{D}_3$
be the set of all graphs constructed by adding arcs of two
directions between $v_i$, $i=0,1$, and every vertex of
$\overleftrightarrow{K}_n$, and possibly, adding any of the arcs
$(v_0, v_1)$ and $(v_1,v_0)$, or both. Finally, let $D_4$ be the
digraph showed in Figure \ref{figure:exception4_D}. Our main
result is as below.

\begin{theorem} \label{MainTheoremDigraph}
Let $D$ be a digraph. For every vertex pair $u$ and $v$, where
there is no arc from $u$ to $v$, we have $d^+(u)+d^-(v)\geq
|D|-1$, then $D$ has a directed Hamilton cycle, unless $D\in
\mathcal{D}_1$, $\mathcal{D}_2$ or $\mathcal{D}_3$, or $D=D_4$.
\end{theorem}

Let $\mathcal{G}_1$ be the class of graphs $G$ constructed by
identifying an edge of one $K_{m+1,m+1}$ and one $K_{n+1,n+1}$,
and $\mathcal{M}_1$ be the set of all perfect matchings of $G$
containing the identified edge. Let $\mathcal{G}_2$ be the class
of graphs $G$, constructed by taking a copy of $(n+1)K_2$ with
bipartition $(B,\ W)$, and an arbitrary bipartite graph $G_2$ with
bipartition $(B_1, \ W_1)$, where $|B_1|=|W_1|=n$, which has at
least one perfect matching, then connecting every vertex in $B$ to
every vertex in $W_1$, and every vertex in $W$ to every vertex in
$B_1$. Furthermore, let $\mathcal{M}_2$ be the set of all perfect
matchings of $G$, containing all the edges in $(n+1)K_2$ (shown
thick in Figure \ref{exception2}). Let $G_3$ be as shown in Figure
\ref{exception3}, and $\mathcal{G}_3$ the set of the graphs $G$
constructed by taking one copy of $K_{n,n}$ with bipartition $(B,
W)$, and connecting every vertex in $B$ to $w_0$ and $w_1$, every
vertex in $W$ to $b_0$ and $b_1$, and possibly, adding any of the
edges $w_0b_1$, $w_1b_0$, or both. Let $\mathcal{M}_3$ be the set
of perfect matchings of $G$, containing the thick edges in $G_3$.
Finally, we let graph $G_4$ be the graph in Figure
\ref{exception4}, and $M_4$ the perfect matching of it, consisting
of the thick edges. We have the following version of our main
theorem.

\begin{theorem} \label{MainTheorem}
Let $G=(W,B)$ be a bipartite graph with a perfect matching $M$,
for every vertex pair $w\in W$ and $b\in B$, where $wb-$, we have
$d(w)+d(b)\geq \nu/2+1$. Then $G$ has an $M$-alternating Hamilton
cycle, unless one of the following holds.

(1) $G \in \mathcal{G}_1$, and $M \in \mathcal{M}_1$.

(2) $G \in \mathcal{G}_2$, and $M \in \mathcal{M}_2$.

(3) $G\in \mathcal{G}_3$, and $M \in \mathcal{M}_3$.

(4) $G=G_4$ and $M=M_4$.
\end{theorem}

Since the two results are equivalent, we only prove Theorem
\ref{MainTheorem} in the next section. Before that, let's say a
few words on the non-existence of $M$-alternating Hamilton cycles
in the four exceptional cases. In Case (1), an $M$-alternating
cycle of $G$ must contain the identified edge, whose endvertices
form a vertex cut of $G$, so $G$ does not have an $M$-alternating
Hamilton cycle. In Case (2), if there is an $M$-alternating
Hamilton cycle $C$ of $G$, then the edges on $C$ that belong to
$M$ must be in $(n+1)K_2$ and $G_2$ alternately, but there is one
more such edge in $(n+1)K_2$, a contradiction. In Case (3), we can
not have an $M$-alternating Hamilton cycle containing both $e_0$
and $e_1$. Finally in Case (4), the non-existence of any
$M$-alternating Hamilton cycle can be verified directly.

\begin{figure}
\centering
\includegraphics[width=0.5\linewidth]{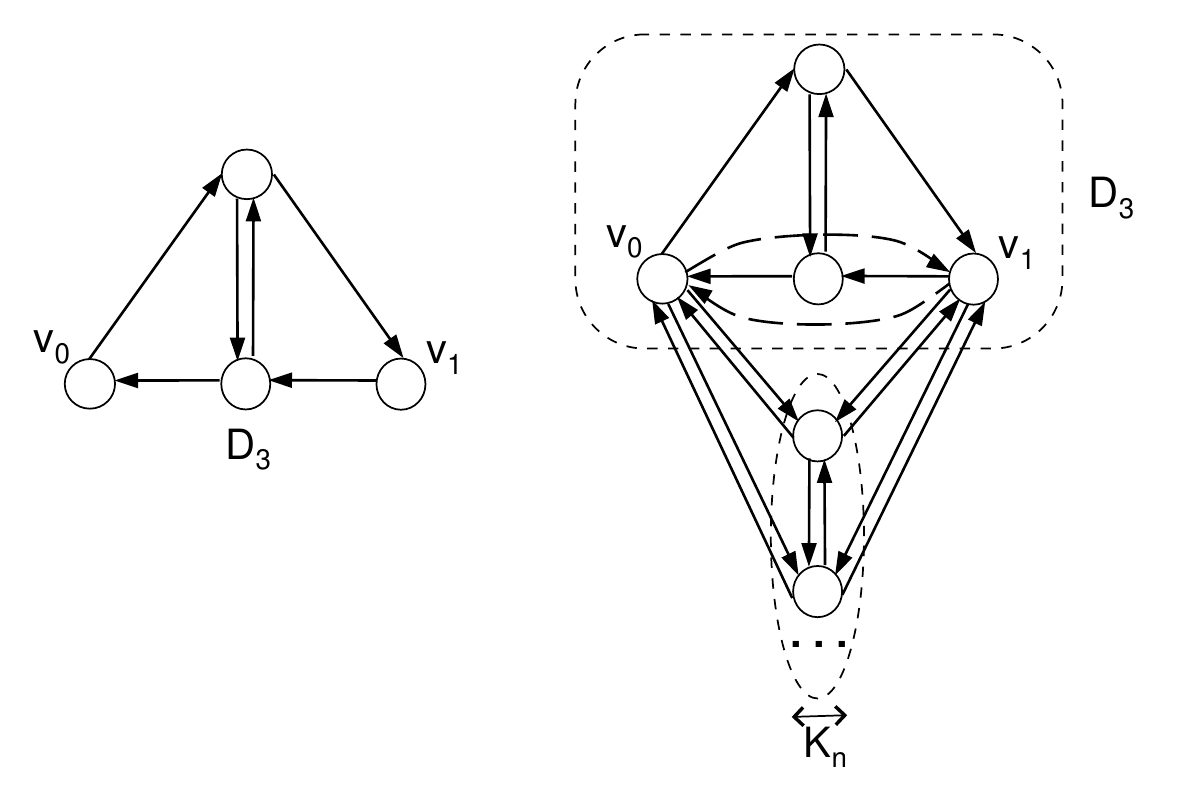}
\caption{Exceptional graph family: $\mathcal{D}_3$}
\label{figure:exception3_D}
\end{figure}

\begin{figure}
\centering
\includegraphics[width=0.33\linewidth]{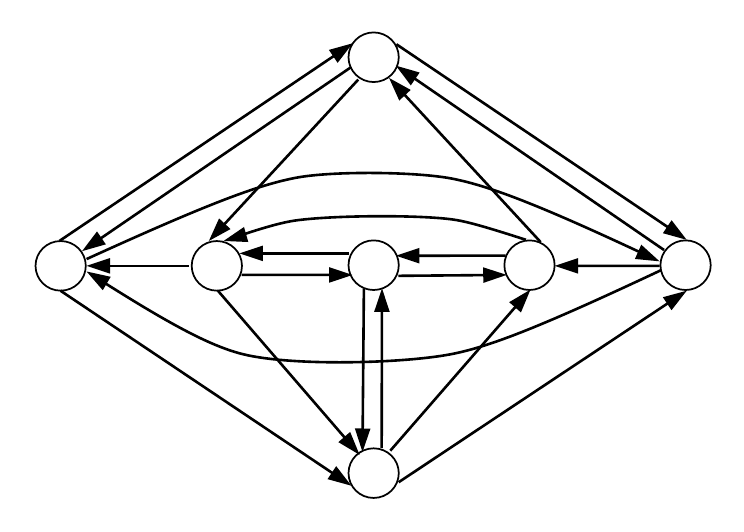}
\caption{Exceptional graph $D_4$} \label{figure:exception4_D}
\end{figure}

\begin{figure}
\centering
\includegraphics[width=0.28\linewidth]{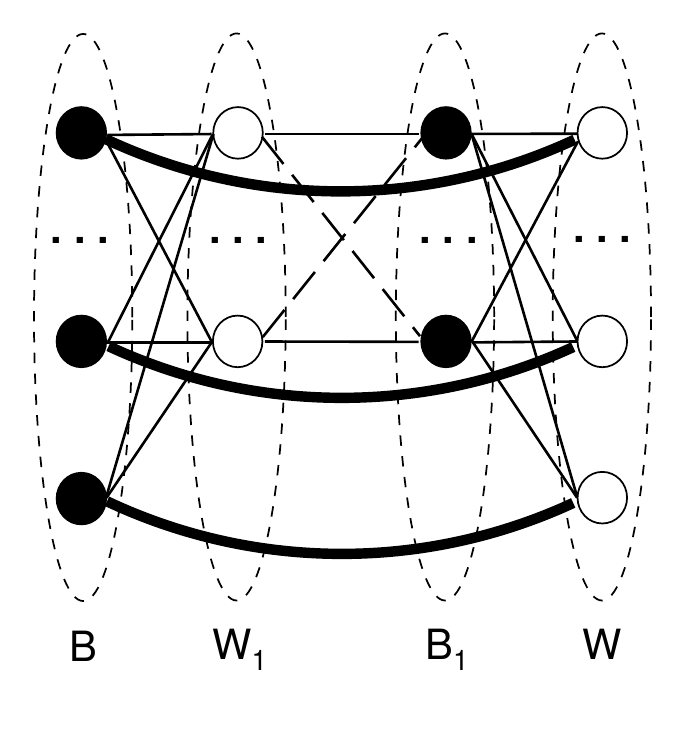}
\caption{Exceptional graph family: $\mathcal{G}_2$}
\label{exception2}
\end{figure}

\begin{figure}
\centering
\includegraphics[width=0.63\linewidth]{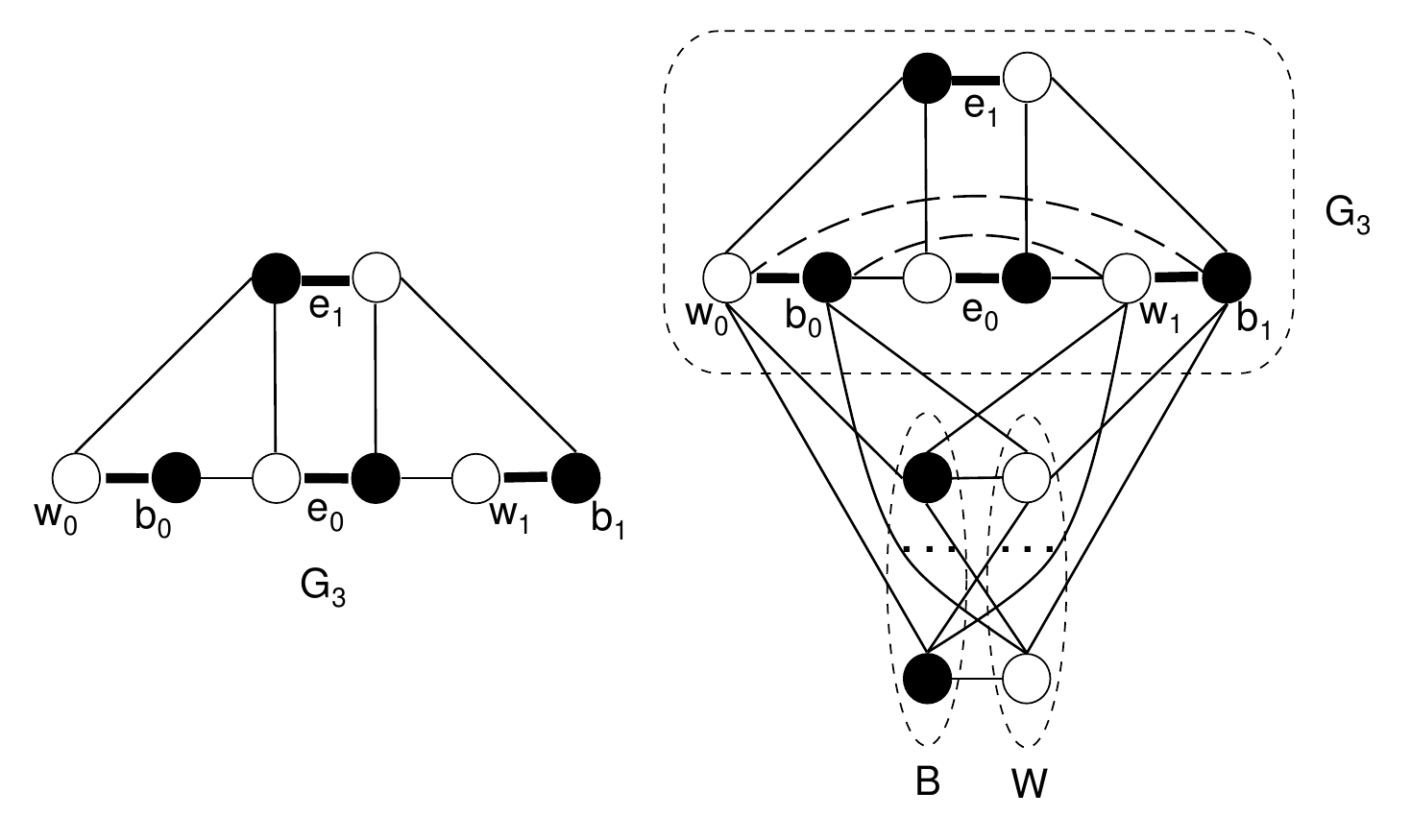}
\caption{Exceptional graph family: $\mathcal{G}_3$}
\label{exception3}
\end{figure}

\begin{figure}
\centering
\includegraphics[width=0.48\linewidth]{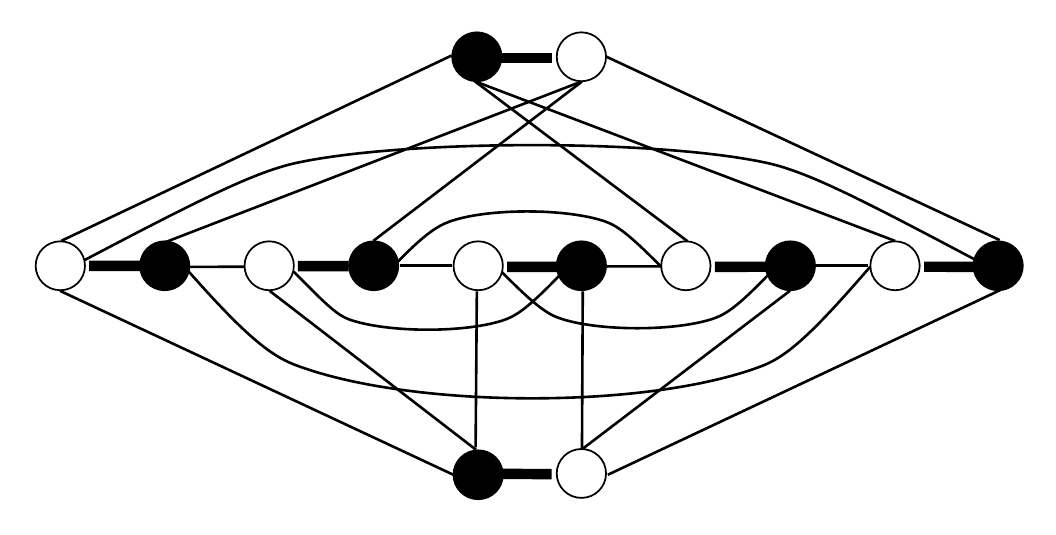}
\caption{Exceptional graph $G_4$} \label{exception4}
\end{figure}

\section{Proof of Theorem \ref{MainTheorem}}
Let $G=(W,B)$ be a bipartite graph satisfying the condition of the
theorem, $M$ a perfect matching of $G$. Suppose that $G$ does not
have an $M$-alternating Hamilton cycle. We prove the theorem by
characterizing $G$. $\newline$

The following two lemmas will be used in our proof.

\begin{lemma}\label{lemma:CycleToPath} Let $G=(W,B)$ be a bipartite graph with a
perfect matching $M$. Let $C=u_{0}u_{1}\ldots u_{2m-1}u_{0}$ be a
longest $M$-alternating cycle in $G$, where $u_{2i}\in W$,
$u_{2i+1}\in B$, and $u_{2i}u_{2i+1}\in M$, $0\leq i\leq m-1$. Let
$b\in B$, $w\in W$ be the ending vertices of a closed $M$-alternating
path $P$ in $G-C$. Then, for every $0\leq i\leq m-1$, either
$u_{2i}b-$ or $u_{2i-1}w-$. Furthermore, if $b\rightarrow C$ and
$w\rightarrow C$, then $|N_C(b)|+|N_C(w)|\leq m-|P|/2+1.$
\end{lemma}

\begin{proof}
If there exists $0\leq k \leq m-1$, such that $u_{2k}b+$ and
$u_{2k-1}w+$, then $u_{2k}C^+u_{2k-1}wPbu_{2k}$ is an
$M$-alternating cycle longer than $C$, a contradiction. Thus, for
$0\leq i\leq m-1$, either $u_{2i}b-$ or $u_{2i-1}w-$.

If $b\rightarrow C$ and $w\rightarrow C$, let $u_{2r}\in N_C(b)$
and $u_{2s-1}\in N_C(w)$ be such that $P^\prime=u_{2s}C^+u_{2r-1}$
is the shortest. Then, there is no neighbor of $w$ and $b$ on
$P^\prime$. Since $C$ is the longest, we have $|P^\prime| \geq
|P|$. So $|N_C(w)|+|N_C(b)|\leq
2+(|C|-|P^\prime|-2)/2=m-|P^\prime|/2+1\leq m-|P|/2+1.$
\end{proof}

\begin{lemma}\label{lemma:CycleToCycle} Let $G$ be a bipartite graph with a
perfect matching $M$. Let $C=u_{0}u_{1}\ldots u_{2m-1}u_{0}$ be a
longest $M$-alternating cycle in $G$, where $u_{2i}u_{2i+1}\in M$,
$0\leq i\leq m-1$. Let $C_1$ be an $M$-alternating cycle in $G-C$.
For any vertex set $\{u_{2i-1},u_{2i}\}$, $0\leq i\leq m-1$,
either $u_{2i-1}\nrightarrow C_1$ or $u_{2i}\nrightarrow C_1$.
\end{lemma}

\begin{proof}
Suppose there exists $0\leq k\leq m-1$ such that
$u_{2k-1}\rightarrow C_1$ and $u_{2k}\rightarrow C_1$. Let $b\in
N_{C_1}(u_{2k})$ and $w\in N_{C_1}(u_{2k-1})$. We can always find
a closed $M$-alternating path, $P$, as a segment of $C_1$,
connecting $b$ and $w$. Then $u_{2k}C^+u_{2k-1}wPbu_{2k}$ is an
$M$-alternating cycle longer than $C$, contradicting our
condition.
\end{proof}

In our proof, some important intermediate results are shown as
claims.

\begin{claim} \label{|G|/2+1}
There is an $M$-alternating cycle in $G$ whose length is at least
$\nu /2+1$.
\end{claim}

\begin{proof}
Let $P=u_0u_1\ldots u_{2p-1}$ be a longest closed
$M$-alternating path in $G$, then, all neighbors of $u_0$ and
$u_{2p-1}$ in $G$ should be on $P$.

If $u_0u_{2p-1}+$, then we obtain a cycle $C=u_0u_1\ldots
u_{2p-1}u_0$. Since $P$ is the longest, $e(V(C),V(G-C))=0$.
However, $G$ is connected, so $C$ must be an $M$-alternating
Hamilton cycle and the claim holds.

If $u_0u_{2p-1}-$, by our condition, $d(u_0)+d(u_{2p-1})\geq
\nu/2+1$. Without lost of generality, assume that $d(u_0)\geq
d(u_{2p-1})$ and let $u_{2i-1}$ be the neighbor of $u_0$ with the
maximum $i$, $1\leq i \leq p$. Then, $i\geq (\nu/2+1)/2$ and $u_0P
u_{2i-1}u_0$ is an $M$-alternating cycle with length at least $2i
\geq \nu/2+1$. This proves our claim.
\end{proof}

Now let $C=u_0u_1\ldots u_{2m-1}u_0$ be a longest $M$-alternating
cycle in $G$, where $u_{2i}\in W$, $u_{2i-1}\in B$ and
$u_{2i}u_{2i+1}\in M$. Let $G_1=G-C$. Denote the neighborhood and
degree of $v\in V(G_1)$ in $G_1$ by $N_1(v)$ and $d_1(v)$. By
Claim \ref{|G|/2+1}, $|G_1|\leq \nu/2-1$.

Let $P_1=v_0v_1\ldots v_{2p_1-1}$ be a longest closed
$M$-alternating path in $G_1$, where $v_{2i}\in W$ and
$v_{2i+1}\in B$, $0\leq i \leq p_1-1$. Then $N_1(v_0)$,
$N_1(v_{2p_1-1})\subseteq V(P_1)$, and $d_1(v_0),\ d_1(v_{2p_1-1})\leq
p_1$. Firstly, we prove that $v_0\rightarrow C$ and $v_{2p_1-1}
\rightarrow C$.

If $v_0 \nrightarrow C$ and $v_{2p_1-1}\nrightarrow C$, then
$d(v_0)+d(v_{2p_1-1})\leq 2p_1 \leq |G_1| \leq \nu/2-1$. By the
condition of our theorem, $v_0v_{2p_1-1}+$, and we get a cycle
$C_1=v_0v_1\ldots v_{2p_1-1}v_0$ in $G_1$. By Lemma
\ref{lemma:CycleToCycle}, for any two vertices $u_{2i-1}$ and $u_{2i}$
on $C$, at least one of them, say $u_{2i}\nrightarrow C_1$. Then
$d(u_{2i})\leq \nu/2-p_1$. But then $d(u_{2i})+d(v_{2p_1-1})\leq
\nu/2$, contradicting the condition of the theorem.

If only one of $v_0$ and $v_{2p_1-1}$, say $v_0\rightarrow C$. Let
a neighbor of $v_0$ on $C$ be $u_{2j-1}$, by Lemma
\ref{lemma:CycleToPath}, $u_{2j}$ sends no edge to $P_1$, so
$d(u_{2j})\leq \nu/2-p_1$, and $d(u_{2j})+d(v_{2p_1-1})\leq
\nu/2$, again contradicting the condition of the theorem.

Therefore $v_0\rightarrow C$ and $v_{2p_1-1}\rightarrow
C$.$\newline$

By Lemma \ref{lemma:CycleToPath},
$|N_C(v_0)|+|N_C(v_{2p_1-1})|\leq m-p_1+1$. Therefore,
\begin{eqnarray} \label{eqn:degreesum}d(v_0)+d(v_{2p_1-1})
&\leq& 2p_1+(m-p_1+1) \nonumber \\
&=& m+p_1+1 \nonumber \\
&\leq& m+|G_1|/2+1 \nonumber \\
&=& \nu/2+1. \end{eqnarray} If $v_0v_{2p_1-1}-$, then by our
condition, $d(v_0)+d(v_{2p_1-1})\geq \nu/2+1$ and hence equalities
in (\ref{eqn:degreesum}) hold. But then we must have
$v_0v_{2p_1-1}+$, a contradiction. So $v_0v_{2p_1-1}+$, and we get
a cycle $C_1=v_0v_1\ldots v_{2p_1-1}v_0$. $\newline$

If $G_1-C_1$ is nonempty, then there exists an edge $wb\in M\cap
E(G_1-C_1)$, where $w\in W$ and $b\in B$. By the choice of $P_1$,
$e(V(C_1), V(G_1-C_1))=0$.  By our condition,
$d(w)+d(b)+d(v_0)+d(v_{2p_1-1})\geq 2(\nu/2+1)=\nu+2$. However, by
Lemma \ref{lemma:CycleToPath}, $|N_C(w)|+|N_C(b)|\leq m$, and
hence $d(w)+d(b)\leq |G_1|-2p_1+m$, while
$d(v_0)+d(v_{2p_1-1})\leq m+p_1+1$ by (\ref{eqn:degreesum}),
therefore $d(w)+d(b)+d(v_0)+d(v_{2p_1-1})\leq |G_1|+2m-p_1+1 =
\nu-p_1+1 < \nu+1$, a contradiction. Hence, $G_1-C_1$ must be
empty, then $|G_1|=2p_1$ and $C_1$ is an $M$-alternating Hamilton
cycle of $G_1$. $\newline$

{
We claim that every vertex of $G_1$ sends some edges to $C$.}
{
Let $v$ be any vertex in $G_1$. Since $G_1$ has an $M$-alternating Hamilton cycle $C_1$, we can choose a closed $M$-alternating Hamilton path $P_1$ of $G_1$ starting from $v$. By above discussion, $v$ sends some edges to $C$.}
$\newline$

For a longest $M$-alternating cycle $C$ in $G$, we call the graph
$G_1=G-C$ a critical graph (with respect to $C$) and a closed
$M$-alternating Hamilton path of $G_1$, $P_1=v_0v_1\ldots
v_{2p_1-1}$, where $v_{2i}\in W$ and $v_{2i+1}\in B$, a critical
path, or a critical edge if $|P_1|=2$. For a critical path $P_1$,
we can always find $u_{2s-1}\in N_C(v_0)$ and $u_{2r}\in
N_C(v_{2p_1-1})$, such that $P_2=u_{2s}C^+u_{2r-1}$ is the
shortest. We let $R=u_{2r}C^+u_{2s-1}$.

By Lemma \ref{lemma:CycleToCycle}, $u_{2s}\nrightarrow G_1$ and
$u_{2r-1}\nrightarrow G_1$. Further, for any edge $u_{2i-1}u_{2i}$
on $R$, we must have $e(\{u_{2i-1},\ u_{2i}\},\ \{u_{2s},\
u_{2r-1}\})\leq 1$, or we get an $M$-alternating Hamilton cycle
$$u_{2r}C^+u_{2i-1}u_{2s}C^+u_{2r-1}u_{2i}C^+u_{2s-1}v_0P_1v_{2p_1-1}u_{2r}.$$
Hence,
\begin{eqnarray}\label{eqn:u2s+u2r-1}d(u_{2s})+d(u_{2r-1})\leq
|P_2|+2+(|R|-2)/2=|P_2|+|R|/2+1.\end{eqnarray}
Moreover,
\begin{eqnarray} \label{eqn:v0+v2p1-1}
d(v_0)+d(v_{2p_1-1})\leq 2p_1+2+(|R|-2)/2=2p_1+|R|/2+1.
\end{eqnarray}
So,
\begin{eqnarray}\label{eqn:sumof4leq}d(u_{2s})+d(u_{2r-1})+d(v_0)+d(v_{2p_1-1})\leq
2p_1+|P_2|+|R|+2=\nu+2.\end{eqnarray} However $v_0u_{2r-1}-$ and
$v_{2p_1-1}u_{2s}-$, by our condition,
\begin{eqnarray}\label{eqn:sumof4geq}d(u_{2s})+d(u_{2r-1})+d(v_0)+d(v_{2p_1-1})\geq 2(\nu/2+1)=\nu+2.\end{eqnarray}
So all equalities in (\ref{eqn:u2s+u2r-1}), (\ref{eqn:v0+v2p1-1}),
(\ref{eqn:sumof4leq}) and (\ref{eqn:sumof4geq}) must hold. To get
equality in (\ref{eqn:v0+v2p1-1}),
$v_0$ (respectively $v_{2p_1-1}$) must be adjacent to all vertices in $V(G_1) \cap B$ (respectively $V(G_1) \cap W$).
and for any edge $u_{2i-1}u_{2i}$ on $R$,
$e(\{u_{2i-1}, u_{2i}\}, \{v_{0},v_{2p_1-1}\})=1$. Therefore, for
a critical path $P_1=v_0v_1\ldots v_{2p_1-1}$, we find two closed
$M$-alternating paths $R$ and $P_2$ as segments of $C$, such that
$V(C)=V(R)\cup V(P_2)$, where the ending vertices of $R$ is adjacent
to $v_0$ and $v_{2p_1-1}$, respectively, and for any edge
$u_{2i-1}u_{2i}\notin M$ on $R$, $e(\{u_{2i-1}, u_{2i}\},
\{v_{0},v_{2p_1-1}\})=1$, while $e(V(P_2),
\{v_{0},v_{2p_1-1}\})=0$. We call $P_2$ the opposite path, and $R$
the central path for $P_1$.

Furthermore, to get equality in (\ref{eqn:u2s+u2r-1}),
$u_{2s}$ (respectively $u_{2r-1}$) must be adjacent to all vertices in $V(P_2)\cap B$ (respectively $V(P_2)\cap W$).
In particular $u_{2s}u_{2r-1}+$.

\begin{claim} \label{claim:G_1 complete}
A critical graph $G_1$ is complete bipartite.
\end{claim}
\begin{proof}

{
Since $C_1$ is an $M$-alternating Hamilton cycle of $G_1$, for any vertex $v\in V(G_1)$, $P_1$ can be chosen so that it is starting from $v$. By the equality of (\ref{eqn:v0+v2p1-1}), $v$ sends edges to every vertex in the opposite part of $G_1$.
}
\end{proof}

Let $G_2=G[V(P_2)]$. We call $G_2$ the opposite graph. We choose
$C$, $G_1$ and $P_1$ so that the opposite path $P_2$ is the
shortest.

\begin{claim} \label{claim:G_1 G_2 disjoint}
$e(V(G_1), V(G_2))=0$, and
$u_{2s-1}$ (respectively $u_{2r}$) is adjacent to every vertex in $V(G_1)\cap W$ (respectively $V(G_1)\cap B$).
\end{claim}
\begin{proof}
If $|G_1|=2$ the conclusion holds. We assume that $|G_1|\geq 4$.

For any closed $M$-alternating Hamilton path $P_1^\prime$ of $G_1$ with ending vertices $w\in W$ and $b\in B$, we can find an opposite path $P_2^\prime$ and a central path $R^\prime$ for $P_1^\prime$. Since $P_2$ is chosen as the shortest, $|P_2^\prime|\geq |P_2|$ and $|R^\prime| \leq |R|$.
Similar to (\ref{eqn:v0+v2p1-1}) we have
\begin{eqnarray} \label{eqn:w+b}
d(w)+d(b)\leq 2p_1+|R^\prime|/2+1 \leq 2p_1+|R|/2+1.
\end{eqnarray}
Together with (\ref{eqn:u2s+u2r-1}), we have
\begin{eqnarray}\label{eqn:sumof4leqwb}d(u_{2s})+d(u_{2r-1})+d(w)+d(b)\leq
\nu+2.\end{eqnarray}

Since $u_{2r}$ and $u_{2s-1}$ send edges to $G_1$, which has an $M$-alternating Hamilton cycle, by Lemma \ref{lemma:CycleToCycle}, $u_{2r-1}\nrightarrow G_1$ and $u_{2s} \nrightarrow G_1$, and hence $wu_{2r-1}-$ and $bu_{2s}-$. By the condition given,
\begin{eqnarray}\label{eqn:sumof4geqwb}d(u_{2s})+d(u_{2r-1})+d(w)+d(b)\geq 2(\nu/2+1)=\nu+2.\end{eqnarray}

Hence all equalities in (\ref{eqn:w+b}), (\ref{eqn:sumof4leqwb}) and (\ref{eqn:sumof4geqwb}) must hold. Therefore $|R|=|R^\prime|$, $|P_2^\prime|=|P_2|$, $d(w)=d(v_0)=\nu/2+1-d(u_{2r-1})$ and $d(b)=d(v_{2p_1-1})=\nu/2+1-d(u_{2s})$. In other words, all opposite paths (respectively all central paths) have the same length. Since any vertex in $G_1$ can be an ending vertex of an $M$-alternating Hamilton path, all vertices in $V(G_1)\cap W$ have the same degree $\nu/2+1-d(u_{2r-1})$, and all vertices in $V(G_1)\cap B$ have the same degree $\nu/2+1-d(u_{2s})$.

Let $b\neq v_{2p_1-1}$ be a vertex in $V(G_1)\cap B$, assume that $b$ has a neighbor $u_{2r^\prime}$ on $P_2$. Since $G_1$ is complete bipartite we can always find a closed $M$-alternating path $P_1^{\prime\prime}$ connecting $v_0$ and $b$ in $G_1$. (Note that $P_1^{\prime\prime}$ need not to be Hamilton. If $b=v_1$, $P_1^{\prime\prime}$ can only be the edge $v_0v_1$.) Let $P_2^{\prime\prime}=u_{2s}C^+u_{2r^\prime-1}$ and $R^{\prime\prime}=u_{2r^\prime} C^+ u_{2s-1}$. For any vertex pair $\{u_{2i-1}, u_{2i}\}$ on the path $R^{\prime\prime}$, we have $e(\{u_{2i-1},u_{2i}\},\{u_{2s},u_{2r^\prime-1}\})\leq 1$, or we get an $M$-alternating cycle
$$u_{2r^\prime}C^+u_{2i-1}u_{2s}C^+u_{2r^\prime-1}u_{2i}C^+u_{2s-1}v_0P_1^{\prime\prime}bu_{2r^\prime},$$
which is longer than $C$.
Therefore, $$d(u_{2s})+d(u_{2r^\prime-1}) \leq |P_2^{\prime\prime}|+2+(|R^{\prime\prime}|-2)/2=|P_2^{\prime\prime}|+|R^{\prime\prime}|/2+1 < |P_2|+|R|/2+1.$$ By $d(v_0)+d(b)=d(v_0)+d(v_{2p_1-1})=2p_1+|R|/2+1$, we have $d(u_{2s})+d(u_{2r^\prime-1})+d(v_0)+d(b) < (|P_2|+|R|/2+1) +(2p_1+|R|/2+1)= \nu+2$. However, since $u_{2s}b-$ and $u_{2r^\prime-1}v_0-$, by our condition, $d(u_{2s})+d(u_{2r^\prime-1})+d(v_0)+d(b) \geq \nu+2$, a contradiction. Hence $b$, and similarly any $w\in V(G_1)\cap W$, must not have any neighbor on $P_2$. That is, $e(V(G_1), V(G_2))=0$.

For any closed $M$-alternating Hamilton path $P_1^\prime$ of $G_1$ with ending vertices $w\in W$ and $b \in B$, let $P_2^\prime$ be an opposite path of it. Since $w$ and $b$ send no edges to $P_2$, $P_2$ must be part of $P_2^\prime$. However, all opposite paths have the same length, so $|P_2^\prime|=|P_2|$, and therefore $P_2^\prime = P_2$. Then, $wu_{2s-1}+$ and $bu_{2r}+$. Since any vertex in $G_1$ can be an ending vertex of a closed $M$-alternating Hamilton path of $G_1$, we prove the second part of the claim.
\end{proof}

\begin{claim} \label{claim:G_2 complete}
$G_2$ is complete bipartite,
and $u_{2s-1}$ (respectively $u_{2r}$) is adjacent to every vertex in $V(G_2)\cap W$ (respectively $V(G_2)\cap B$).
\end{claim}
\begin{proof}

By above discussions, $u_{2s}u_{2r-1}+$ and we have a cycle
$C_2=u_{2s}C^+u_{2r-1}u_{2s}$. Since $e(V(G_1), V(G_2))=0$, for
every edge $u_{2j-1}u_{2j}$ on $P_2$, where $s+1\leq j \leq r-1$,
we can replace $u_{2r-1}$ with $u_{2j-1}$ and $u_{2s}$ with
$u_{2j}$ in (\ref{eqn:u2s+u2r-1}), (\ref{eqn:sumof4leq}) and
(\ref{eqn:sumof4geq}), and all equalities must hold. So,
$u_{2j-1}$ (respectively $u_{2j}$) must be adjacent to all vertices in $V(P_2)\cap W$ (respectively $V(P_2)\cap B$),
$u_{2j-1}u_{2r}+$ and
$u_{2j}u_{2s-1}+$, therefore the claim holds.
\end{proof}

For convenience we change some notations henceforth. We let
$|G_2|=2p_2$ and the vertices of $G_2$ be $v_0^\prime$,
$v_1^\prime$, \ldots, $v_{2p_2-1}^\prime$, where $v_{2j}^\prime
v_{2j+1}^\prime\in M$, for $0\leq j\leq p_2-1$, and let
$R=u_0u_1\ldots u_{2r-1}$.

Now we discuss the situations case by case, with respect to the length of $R$
and the distribution of edges between $R$ and $G_i$, $i=1,2$.

\noindent\emph{\textbf{Case 1}}. $|R|=2$.

Then $R=u_0u_1$. By Claim \ref{claim:G_1 G_2 disjoint} and Claim
\ref{claim:G_2 complete},  For any $0\leq i\leq p_1-1$ and $0\leq
j \leq p_2-1$, $u_0v_{2i+1}+$, $u_0v_{2j+1}^\prime+$, $u_1v_{2i}+$
and $u_1v_{2j}^\prime+$. Therefore $G\in \mathcal{G}_1$ and $M\in
\mathcal{M}_1$.

\noindent\emph{\textbf{Case 2}}. $|R|\geq 4$.

\begin{claim}\label{claim:toG1G2}
For $j=1,\ 2$, and every edge $u_{2i-1}u_{2i}$, $1\leq i \leq
r-1$, exactly one of $u_{2i-1}\rightarrow G_j$ and
$u_{2i}\rightarrow G_j$ holds. Furthermore, if
$u_{2i-1}\rightarrow G_j$ (respectively $u_{2i}\rightarrow G_j$),
it is adjacent to all vertices in $V(G_j)\cap W$ (respectively $V(G_j)\cap B$).
\end{claim}
\begin{proof}
Firstly, we prove that for $j=1,\ 2$ and every edge
$u_{2i-1}u_{2i}$, $1\leq i \leq r-1$, $u_{2i-1}\nrightarrow G_j$
or $u_{2i}\nrightarrow G_j$. By Lemma \ref{lemma:CycleToCycle}, the
conclusion holds for $G_1$. Now we prove it for $G_2$. Suppose to
the contrary that there exists $1\leq l \leq r-1$ such that
$u_{2l-1}\rightarrow G_2$ and $u_{2l} \rightarrow G_2$, and let
$v_{2s}^\prime \in N_{G_2}(u_{2l-1})$ and $v_{2t+1}^\prime \in
N_{G_2}(u_{2l})$. If $|G_2|=2$ or $t\neq s$, We can find a closed
$M$-alternating Hamilton path $Q$ of $G_2$ connecting
$v_{2s}^\prime$ and $v_{2t-1}^\prime$, and hence we have an
$M$-alternating Hamilton cycle
$$u_0Ru_{2l-1}v_{2s}^\prime Q v_{2t-1}^\prime
u_{2l}Ru_{2r-1}v_0P_1v_{2p_1-1}u_0$$ of $G$, contradicting our
assumption. If $|G_2|\geq 4$ and $t=s$, let $P_2^\prime$ be a
closed $M$-alternating Hamilton path of $G_2-\{v_{2s}^\prime,\
v_{2s+1}^\prime\}$. Then $P_2^\prime$ is an opposite path for
$P_1$, with the central path $u_0Ru_{2l-1}v_{2s}^\prime
v_{2s+1}^\prime u_{2l}Ru_{2r-1}$, which is shorter than $P_2$,
contradicting our choice of $P_2$. Hence $u_{2i-1}\nrightarrow
G_2$ or $u_{2i}\nrightarrow G_2$, for  $1\leq i \leq r-1$.

Arbitrarily choose $0\leq l \leq p_1-1$ and $0\leq k \leq p_2-1$.
We have $d(v_{2l})+d(v_{2l+1})\leq 2p_1+2+(|R|-2)/2=2p_1+r+1$ and
similarly $d(v_{2k}^\prime)+d(v_{2k+1}^\prime)\leq 2p_2+r+1$. So
\begin{eqnarray} \label{eqn:k+l_leq}d(v_{2l})+d(v_{2l+1})+d(v_{2k}^\prime)+d(v_{2k+1}^\prime)\leq
2p_1+2p_2+2r+2=\nu+2.\end{eqnarray} However
$v_{2l}v^\prime_{2k+1}-$ and $v_{2l+1}v^\prime_{2k}-$, by the
condition of the theorem,
\begin{eqnarray} \label{eqn:k+l_geq}d(v_{2l})+d(v_{2k+1}^\prime)+d(v_{2l+1})+d(v_{2k}^\prime)
\geq 2(\nu/2+1)=\nu+2,\end{eqnarray} and all equalities must hold.
To obtain equalities, for $j=1,\ 2$, and every edge
$u_{2i-1}u_{2i}$, $1\leq i \leq r-1$, exactly one of $u_{2i-1}
\rightarrow G_j$ and $u_{2i}\rightarrow G_j$ must hold.
Furthermore, since $l$ and $k$ are arbitrarily chosen, we prove
that if
$u_{2i-1}\rightarrow G_j$ (respectively $u_{2i}\rightarrow G_j$), it
is adjacent to all vertices in $V(G_j)\cap W$ (respectively $V(G_j)\cap B$).
\end{proof}

Let $1\leq i \leq r-1$. We define $E_1$ ($E_1^\prime$) to be the
set of edges $u_{2i-1}u_{2i}$, where $u_{2i-1}v_{2j}+$, for every
$0\leq j\leq p_1-1$ ($u_{2i-1}v_{2k}^\prime+$, for every $0\leq k
\leq p_2-1$), and $E_2$ ($E_2^\prime$) to be the set of edges
$u_{2i-1}u_{2i}$, where $u_{2i}v_{2j+1}+$, for every $0\leq j\leq
p_1-1$ ($u_{2i}v_{2k+1}^\prime+$, for every $0\leq k\leq p_2-1$).

By Claim \ref{claim:toG1G2}, for every $1\leq i \leq r-1$,
$u_{2i-1}u_{2i}\in E_1\cap E_1^\prime$, $E_1\cap E_2^\prime$,
$E_2\cap E_1^\prime$ or $E_2\cap E_2^\prime$. Accordingly, we say
that $u_{2i-1}u_{2i}$ is an edge of type I, II, III or IV for
$G_1$, $G_2$ and $R$. Let the number of edges $u_{2i-1}u_{2i}$
belonging to $E_1\cap E_1^\prime$, $E_1\cap E_2^\prime$, $E_2\cap
E_1^\prime$ and $E_2\cap E_2^\prime$ be $t_{11}$, $t_{12}$,
$t_{21}$ and $t_{22}$, respectively. We have
$d(v_0)=t_{11}+t_{12}+p_1+1$, $d(v_1)=t_{22}+t_{21}+p_1+1$,
$d(v_0^\prime)=t_{11}+t_{21}+p_2+1$ and
$d(v_1^\prime)=t_{22}+t_{12}+p_2+1$.

Since equalities hold in (\ref{eqn:k+l_leq}) and
(\ref{eqn:k+l_geq}), we have
$d(v_{2l})+d(v_{2k+1}^\prime)=d(v_{2l+1})+d(v_{2k}^\prime)=\nu/2+1$
for any $0\leq l\leq p_1-1$ and $0\leq k \leq p_2-1$, Hence
\begin{eqnarray} \label{eqn:t12=t21}
t_{11}+t_{22}+2t_{12}+p_1+p_2+2 &=& d(v_0)+d(v_1^\prime) \nonumber
\\ &=& \nu/2+1 \nonumber \\ &=& d(v_1)+d(v_0^\prime) \nonumber \\ &=&
t_{11}+t_{22}+2t_{21}+p_1+p_2+2.
\end{eqnarray} So $t_{12}=t_{21}$.

We let $t_1=t_{11}$, $t_2=t_{22}$ and $t_{0}=t_{12}=t_{21}$, then
$t_1+t_2+2t_0=r-1$. $\newline$

We summarise some structural results in the form of observations.

\begin{observation} \label{observation:E2E1}
If there exists $1\leq j<i\leq r-1$, such that $u_{2i-1}u_{2i} \in
E_1$ ($E_1^\prime$) and $u_{2j-1}u_{2j}\in E_2^\prime$ ($E_2$).
Then $u_{2j-1}u_{2i}-$.
\end{observation}
\begin{proof}
If $u_{2j-1}u_{2i}+$, we obtain an $M$-alternating Hamilton cycle
$$u_0Ru_{2j-1}u_{2i}Ru_{2r-1}v_0^\prime P_2v_{2p_2-1}^\prime u_{2j}Ru_{2i-1}v_0P_1v_{2p_1-1}u_0$$
$$(u_0Ru_{2j-1}u_{2i}Ru_{2r-1}v_0P_1v_{2p_1-1}u_{2j}Ru_{2i-1}v_0^\prime
P_2v_{2p_2-1}^\prime u_0),$$ contradicting our assumption.
\end{proof}

\begin{observation}\label{observation:CriticalEdge}
If there exists $1\leq i \leq r-2$, such that $u_{2i-1}u_{2i} \in
E_1$ and $u_{2i+1}u_{2i+2}\in E_2$, then $u_{2i}u_{2i+1}$ is a
critical edge, $|G_1|=|G_2|=2$, and exactly one of
$u_{2i}v_1^\prime+$ and $u_{2i+1}u_0+$ ($u_{2i+1}v_0^\prime+$ and
$u_{2i}u_{2r-1}+$) holds.

If there exists $1\leq i \leq r-2$, such that $u_{2i-1}u_{2i} \in
E_1^\prime$ and $u_{2i+1}u_{2i+2}\in E_2^\prime$, then
$u_{2i}u_{2i+1}$ is a critical edge, $|G_1|=2$, and exactly one of
$u_{2i}v_1+$ and $u_{2i+1}u_0+$ ($u_{2i+1}v_0+$ and
$u_{2i}u_{2r-1}+$) holds.
\end{observation}

\begin{proof}
Suppose there exists $1\leq i\leq r-2$, such that
$u_{2i-1}u_{2i}\in E_1$ and $u_{2i+1}u_{2i+2}\in E_2$, then
$u_{2i}u_{2i+1}$ is a critical edge with respect to the
$M$-alternating cycle
$$u_0Ru_{2i-1}v_0P_1v_{2p_1-1}u_{2i+2}Ru_{2r-1}v_0^\prime
P_2v_{2p_2-1}^\prime u_0,$$
where $P_1$ is an opposite path. Since
$G_1$ is critical, $|G_1|=2$. Since $|P_1|=2$, and $P_2$ is the
shortest opposite path, $|G_2|=2$. Since $u_0v_1^\prime$
($u_{2r-1}v_0^\prime$) are on a central path for the critical edge
$u_{2i}u_{2i+1}$ and the opposite path $v_0v_1$, exactly one of
$u_{2i+1}u_0+$ and $u_{2i}v_{1}^\prime+$ ($u_{2i+1}v_0^\prime+$
and $u_{2i}u_{2r-1}+$) holds.

{
Now suppose there exists $1\leq i\leq r-2$, such that
$u_{2i-1}u_{2i}\in E_1^\prime$ and $u_{2i+1}u_{2i+2}\in E_2^\prime$. Then
$u_{2i}u_{2i+1}$ is a critical edge with respect to the
$M$-alternating cycle
$$u_0Ru_{2i-1} v_0^\prime
P_2v_{2p_2-1}^\prime u_{2i+2}Ru_{2r-1} v_0P_1v_{2p_1-1} u_0,$$
where $P_2$ is an opposite path. Since $G_1$ is critical, $|G_1|=2$. Since $u_0v_1$
($u_{2r-1}v_0$) are on a central path for the critical edge
$u_{2i}u_{2i+1}$ and the opposite path $P_2$, exactly one of
$u_{2i+1}u_0+$ and $u_{2i}v_{1}+$ ($u_{2i+1}v_0+$
and $u_{2i}u_{2r-1}+$) holds.
}
\end{proof}

\begin{observation}\label{observation:E1E2}
If there exists $1\leq i<k<j\leq r-1$, such that
$u_{2i-1}u_{2i}\in E_1$ ($E_1^\prime$), $u_{2j-1}u_{2j}\in E_2$
($E_2^\prime$), $u_{2k-1}u_{2k}\in E_2^\prime$ ($E_2$) and
$u_{2k-1}u_0+$, then $u_{2i}u_{2j-1}-$.
\end{observation}
\begin{proof}
If $u_{2i}u_{2j-1}+$, we obtain an $M$-alternating Hamilton cycle
$$u_0Ru_{2i-1}v_0P_1v_{2p_1-1}u_{2j}Ru_{2r-1}v_0^\prime P_2v_{2p_2-1}^\prime
u_{2k}Ru_{2j-1}u_{2i}Ru_{2k-1}u_0,$$ contradicting our assumption.

By symmetry, the claim holds under the other situation.
\end{proof}

\begin{claim}
$|G_1|=2$.
\end{claim}
\begin{proof}
Suppose $|G_1|\geq 4$. By Observation
\ref{observation:CriticalEdge}, there does not exist $1\leq i\leq
r-1$, such that $u_{2i-1}u_{2i}\in E_1$ ($E_1^\prime$) and
$u_{2i+1}u_{2i+2}\in E_2$ ($E_2^\prime$). Therefore, there can not
exist $i<j$, such that $u_{2i-1}u_{2i}\in E_1$ ($E_1^\prime$) and
$u_{2j-1}u_{2j}\in E_2$ $(E_2^\prime$). In other words, there
exits an integer $0\leq k_1 \leq r-1$ ($0\leq k_2\leq r-1$), such
that for all $i\leq k_1$ ($j\leq k_2$), $u_{2i-1}u_{2i}\in E_2$
($u_{2j-1}u_{2j}\in E_2^\prime$) and for all $i> k_1$ ($j>k_2$),
$u_{2i-1}u_{2i}\in E_1$ ($u_{2j-1}u_{2j}\in E_1^\prime$). It is
easily seen that $t_0=0$ and $k_1=k_2$. We let $k=k_1=k_2$.

Suppose that $t_1,\ t_2 \neq 0$, or equally, $1\leq k\leq r-2$.
Consider the vertices $u_{2k-1}$ and $u_{2k+2}$.
By Observation \ref{observation:E2E1}, for all $j\geq k+1$,
$u_{2k-1}u_{2j}-$, and for all $j\leq k$, $u_{2k+2}u_{2j-1}-$.
Particularly, $u_{2k-1}u_{2k+2}-$. But then we have
$d(u_{2k-1})\leq k+1$, $d(u_{2k+2})\leq r-k$ and
$d(u_{2k-1})+d(u_{2k+2})\leq r+1 < \nu/2+1$, contradicting our
condition.

Suppose one of $t_1$ and $t_2$, say $t_1=0$. Then for $1\leq i
\leq r-1$, $d(u_{2i-1})\leq r$. Moreover $d(v_0)=p_1+1$, so
$d(u_{2i-1})+d(v_0)\leq r+p_1+1 <\nu/2+1$ but $v_0u_{2i-1}-$, a
contradiction.

So we must have $|G_1|=2$.
\end{proof}

\begin{claim} \label{claim:t0=0_or t1=t2=0}
Either $t_0=0$, or $t_1 =t_2= 0$.
\end{claim}
\begin{proof}
Suppose that $t_0>0$, and one of $t_1$ and $t_2$ is greater than
$0$. Without lost of generality, we may assume that $t_1\geq t_2$,
and so $t_1>0$.

Let $u_{2i-1}u_{2i}\in E_1\cap E_1^\prime$, $1\leq i\leq r-1$, be
such that $i$ is the maximum. Then by our condition,
$d(u_{2i})+d(v_1) \geq \nu/2+1$. Hence, $d(u_{2i})\geq
\nu/2+1-d(v_1) = \nu/2+1-(t_2+t_0+2) = t_1+t_0+\nu/2-r$. By
Observation \ref{observation:E2E1}, $u_{2i}$ can not be adjacent
to any $u_{2j-1}$, where $u_{2j-1}u_{2j}\in E_2\cup E_2^\prime$ and
$j<i$. Hence $u_{2i}$ sends at least
$t_1+t_0+\nu/2-r-(t_1+1)=t_0+\nu/2-r-1$ edges to
$\{u_{2r-1}\}\cup\{u_{2j-1}: u_{2j-1}u_{2j}\in E_2\cup
E_2^\prime,\  j>i+1\}$. Since $t_0>0$ and $\nu/2-r\geq 2$,
$u_{2i}\rightarrow \{u_{2j-1}: u_{2j-1}u_{2j}\in E_2\cup
E_2^\prime,\ j>i+1\}$, so there exists at least one
$u_{2j-1}u_{2j}$ such that $j>i+1$ and $u_{2j-1}u_{2j}\in E_2\cup
E_2^\prime$.

{By our choice of $u_{2i-1}u_{2i}$, $u_{2i+1}u_{2i+2}\in E_2\cup
E_2^\prime$. If $u_{2i+1}u_{2i+2} \in E_2$, then by Observation \ref{observation:CriticalEdge},
$u_{2i}u_{2i+1}$ is a critical edge, and exactly one of $u_{2i}v_{1}^\prime+$ and $u_{2i+1}u_0+$ holds. By
$u_{2i-1}u_{2i}\in E_1^\prime$ we have $u_{2i}v_1^\prime-$, therefore $u_{2i+1}u_0+$. If $u_{2i+1}u_{2i+2} \in E_2^\prime$, then again by Observation \ref{observation:CriticalEdge},
$u_{2i}u_{2i+1}$ is a critical edge, and exactly one of $u_{2i}v_{1}+$ and $u_{2i+1}u_0+$ holds. By
$u_{2i-1}u_{2i}\in E_1$ we have
$u_{2i}v_1-$, hence $u_{2i+1}u_0+$.}

Now we discuss different situations of $u_{2i+1}u_{2i+2}$.

If $u_{2i+1}u_{2i+2}\in E_2\cap E_2^\prime$, let $j>i+1$ be such
that $u_{2i}u_{2j-1}+$, $u_{2j-1}u_{2j}\in E_2\cup E_2^\prime$. By
Observation \ref{observation:E1E2}, $u_{2i}u_{2j-1}-$, a
contradiction.

If $u_{2i+1}u_{2i+2}\in E_1\cap E_2^\prime$ or $E_2\cap
E_1^\prime$, without lost of generality, we may assume that
$u_{2i+1}u_{2i+2}\in E_1\cap E_2^\prime$. Since $u_{2i}u_{2i+1}$
is a critical edge and $u_{2i+1}v_0+$, by Observation
\ref{observation:CriticalEdge}, we have $u_{2i}u_{2r-1}-$. For
$j>i+1$, where $u_{2j-1}u_{2j}\in E_2$, by Observation
\ref{observation:E1E2}, $u_{2i}u_{2j-1}-$. Therefore $u_{2i}$
sends at least $t_0+\nu/2-r-1\geq t_0+1$ edges to $\{ u_{2j-1}:
u_{2j-1}u_{2j}\in E_1\cap E_2^\prime, j>i+1\}$. However, the
number of such $u_{2j-1}$ is at most $t_0$, a contradiction.
\end{proof}

\noindent\emph{\textbf{Case 2.1}}. $t_0=0$.

Without lost of generality, we may assume that $t_1>0$, and let
$u_{2i-1}u_{2i}\in E_1\cap E_1^\prime$.

If there exists $u_{2j-1}u_{2j}$, $j<i$, such that
$u_{2j-1}u_{2j}\in E_2\cap E_2^\prime$, then $u_{2j-1}u_{2i}-$ by
Observation \ref{observation:E2E1}.

If there exists $u_{2j-1}u_{2j}$, $j>i+1$, such that
$u_{2j-1}u_{2j}\in E_2\cap E_2^\prime$, then there exists $i\leq k
\leq j-1$, such that $u_{2k-1}u_{2k}\in E_1\cap E_1^\prime$ and
$u_{2k+1}u_{2k+2}\in E_2\cap E_2^\prime$. By Observation
\ref{observation:CriticalEdge}, $u_{2k}u_{2k+1}$ is a critical
edge, and since $u_{2k+1}v_0-$ and $u_{2k}v_1-$, we have
$u_{2k}u_{2r-1}+$ and $u_{2k+1}u_0+$. By Observation
\ref{observation:E1E2}, $u_{2i}u_{2j-1}-$.

Hence, for all $u_{2j-1}u_{2j}\in E_2\cap E_2^\prime$, $j\neq
i+1$, $u_{2i}u_{2j-1}-$. So, $d(u_{2i})\leq t_1+2$. But then
\begin{eqnarray}\label{OrderG2}\nu/2+1\leq d(u_{2i})+d(v_1)\leq t_1+2+t_2+2 =
(\nu-2p_2-2-2)/2+4=\nu/2-p_2+2.
\end{eqnarray}
Since $p_2\geq 1$, all equalities must hold, hence $p_2=1$ and
$2r-1=\nu-5$. Furthermore, to get $d(u_{2i})= t_1+2$, we must have
the following.

(a) $u_{2i+1}u_{2i+2}\in E_2\cap E_2^\prime$, hence
$u_{2i-1}u_{2i}\neq u_{\nu-7}u_{\nu-6}$.

(b) $u_{2i}u_{2j-1}+$, for all $u_{2j-1}u_{2j}\in E_1\cap
E_1^\prime$.

(c) $u_{2i}u_{\nu-5}+$.

By (a), $t_2\geq 0$, and similarly, for any $u_{2i-1}u_{2i}\in
E_2\cap E_2^\prime$, we can prove the following.

(d) $u_{2i-3}u_{2i-2}\in E_1\cap E_1^\prime$, hence
$u_{2i-1}u_{2i} \neq u_{1}u_2$.

(e) $u_{2i-1}u_{2j}+$, for all $u_{2j-1}u_{2j}\in E_2\cap
E_2^\prime$.

(f) $u_{2i-1}u_{0}+$.

So, the edges $u_{2i-1}u_{2i}$, $1\leq i \leq \nu/2-3$, belong to
$E_1\cap E_1^\prime$ and $E_2\cap E_2^\prime$ alternatively.
Moreover, $u_1u_2\in E_1\cap E_1^\prime$ and
$u_{\nu-7}u_{\nu-6}\in E_2\cap E_2^\prime$. Hence we must have
$\nu=4n+2$, for some integer $n\geq 2$, $u_{4j+1}u_{4j+2}\in
E_1\cap E_1^\prime$ and $u_{4j+3}u_{4j+4}\in E_2\cap E_2^\prime$
for $0\leq j \leq n-2$. The vertex set $\{u_{4j+1},\ u_{4j+2}:
0\leq j \leq n-2\}\cup \{v_0,\ v_0^\prime,\ u_{4n-3}\}$, as well
as $\{u_{4j+3},\ u_{4j+4}: 0\leq j \leq n-2\}\cup \{v_1,\
v_1^\prime,\ u_{0}\}$, induce complete bipartite subgraphs,
respectively.

Let $B_1=\{u_{4j+1}: 0\leq j \leq n-1\}$, $W=\{u_{4j+2}: 0\leq j
\leq n-2\}\cup \{v_0,\ v_0\prime\}$, $B=\{u_{4j+3}: 0\leq j \leq
n-2\}\cup \{v_1, v_1^\prime\}$ and $W_1=\{u_{4j}: 0\leq j \leq
n-1\}$. By above discussion, there can be no more edge between $B$
and $W$. But we can add edges between $B_1$ and $W_1$ freely, to
obtain all graphs $G\in\mathcal{G}_2$, with $M\in \mathcal{M}_2$.
$\newline$

\noindent\emph{\textbf{Case 2.2}}. $t_1=t_2=0$.
{Since $t_1+t_2+2t_0=r-1$, we have $r=2t_0+1$ and
$r$ must be odd.}

If there exists $1\leq i \leq r-2$, such that $u_{2i-1}u_{2i}\in
E_1\cap E_2^\prime$ and $u_{2i+1}u_{2i+2}\in E_2\cap E_1^\prime$
($u_{2i-1}u_{2i}\in E_2\cap E_1^\prime$ and $u_{2i+1}u_{2i+2}\in
E_1\cap E_2^\prime$), we say that an A-change (B-change) occurs at
$u_{2i-1}$. If there exist $i$ and $j$, such that $2\leq i+1<
j\leq r-2$, and there is an A-change (B-change) occurs at
$u_{2i-1}$ and a B-change (A-change) occurs at $u_{2j-1}$, we say
that a change couple occurs at $(u_{2i-1},\ u_{2j-1})$.

\noindent\emph{\textbf{Case 2.2.1}}. $|G_2|\geq 4$.

There can not be any A-change, or by Observation
\ref{observation:CriticalEdge}, $|G_1|=|G_2|=2$. To avoid any
A-change, for $1\leq i \leq (r-1)/2$, $u_{2i-1}u_{2i}\in E_2\cap
E_1^\prime$ and for $(r+1)/2\leq i \leq r-1$, $u_{2i-1}u_{2i}\in
E_1\cap E_2^\prime$.

Suppose that $r=3$. It is not hard to see that $u_0u_3-$ and
$u_2u_5-$, while each of $u_0u_5$ and $u_1u_4$ can be exist or
not. Hence we obtain all the graph in class $\mathcal{G}_3$,
except those with $n=1$.

If $r \geq 5$, then $u_{r-1}u_r$ is a critical edge, with central
path $u_{r+1}Ru_{2r-1}v_0v_1u_0Ru_{r-2}$ and opposite graph $G_2$
(Figure \ref{figure:N11N}). Consider the edge $v_1u_0$ and
$u_1u_2$ on the central path. We have $v_1u_{r-1}+$,
$u_0\rightarrow G_2$, $u_1 \rightarrow G_2$, and by Claim
\ref{claim:t0=0_or t1=t2=0}, $u_2u_r+$. But then an A-change
occurs at $v_1$, a contradiction.


\begin{figure}[!htbp]
\centering
\includegraphics[width=0.55\linewidth]{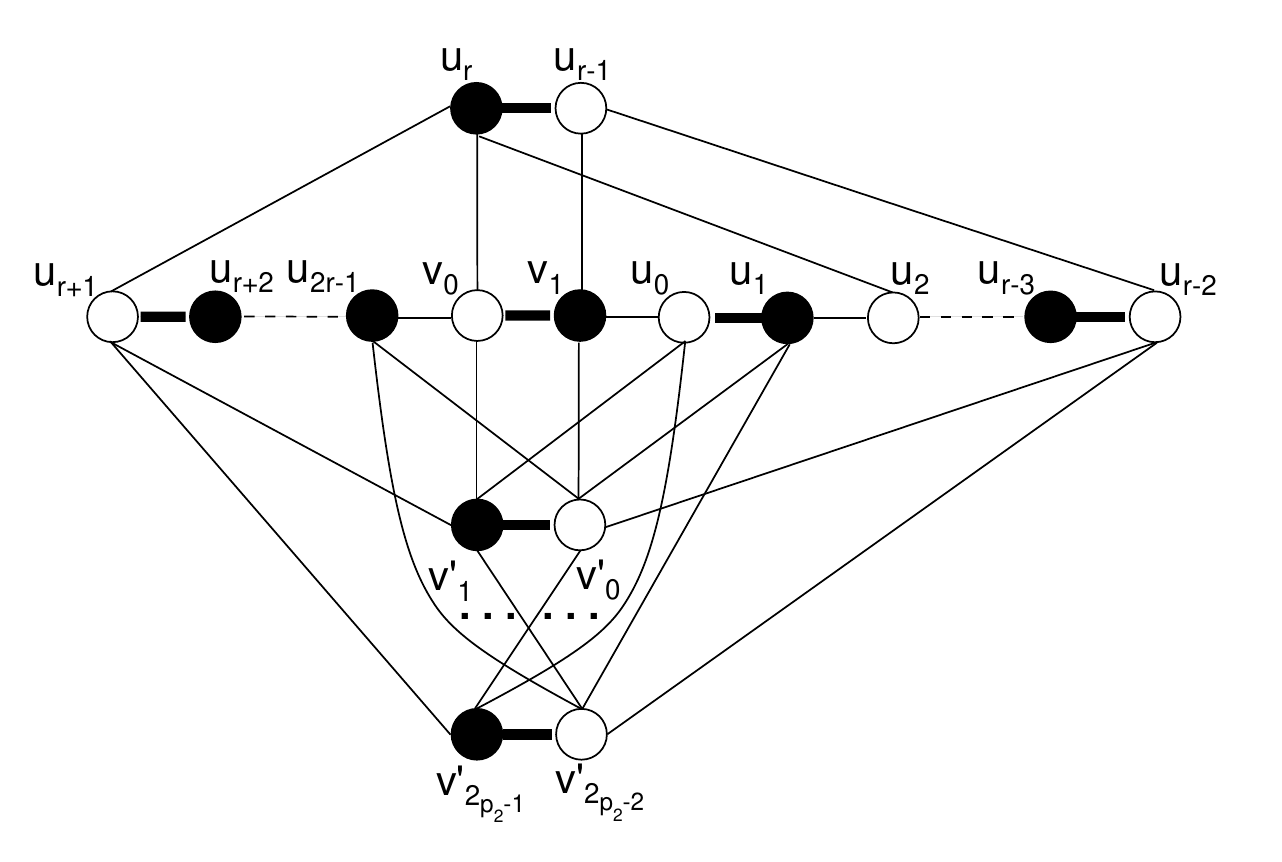}
\caption{Contradiction in Case 2.2.1} \label{figure:N11N}
\end{figure}


\noindent\emph{\textbf{Case 2.2.2}}. $|G_2|=2$.

Then $\nu=4n+6$, for some $n\geq 1$. For $n=1$, it is not hard to
verify that $G\in \mathcal{G}_3$, $M\in \mathcal{M}_3$, and we
obtain all graphs in $\mathcal{G}_3$ together with Case 2.2.1. For
$n=2$, it can be checked that $G=G_4$, $M=M_4$. Henceforth we
assume that $n\geq 3$, and then $r=2n+1 \geq 7$.

We call $G_1$ and $G_2$ a critical edge pair with central path
$R$. Since we have discussed all other cases, we may assume that
for every critical edge pair and the central path, every edge of
the central path that is not in $M$ is of type II or III.

Let there be a change couple occurs at $(u_{2i-1},\ u_{2j-1})$.
Without lost of generality, suppose that an A-change occurs at
$u_{2i-1}$ and a B-change occurs at $u_{2j-1}$, then
$u_{2i}u_{2i+1}$ and $u_{2j}u_{2j+1}$ are critical edges. Since
$u_{2i}u_{2i+1}$ and $v_1v_0$ is a critical edge pair, with the
central path $u_{2i+2}Ru_{2r-1}v_0^\prime v_1^\prime
u_0Ru_{2i-1}$, by our assumption, $u_{2j-1}u_{2j}$ and
$u_{2j+1}u_{2j+2}$ are of type II or III. By $u_{2j}v_1+$ and
$u_{2j+1}v_0+$, we have $u_{2j-1}u_{2i}+$ and $u_{2j+2}u_{2i+1}+$.
Similarly, we have $u_{2i-1}u_{2j}+$ and $u_{2i+2}u_{2j+1}+$.
However, we get an $M$-alternating Hamilton cycle
$$u_0Ru_{2i-1}u_{2j}u_{2j+1}u_{2i+2}Ru_{2j-1}u_{2i}u_{2i+1}v_0^\prime v_1^\prime
u_{2j+2}Ru_{2r-1}v_0v_1u_0$$ then, a contradiction. Therefore,
there must not be any change couple. $\newline$

By symmetry, we may assume that $u_1u_2\in E_1\cap E_2^\prime$,
and let $r_0>0$, $r_1>r_0$ and $r_2\geq r_1$ be such that
$u_1u_2$, $\ldots$, $u_{2r_0-1}u_{2r_0} \in E_1\cap E_2^\prime$,
$u_{2r_0+1}u_{2r_0+2}$, $\ldots$, $u_{2r_1-1}u_{2r_1}\in E_2\cap
E_1^\prime$, $u_{2r_1+1}u_{2r_1+2}$, $\ldots$,
$u_{2r_2-1}u_{2r_2}\in E_1\cap E_2^\prime$ and if
$u_{2r_2+1}u_{2r_2+2}$ exists, $u_{2r_2+1}u_{2r_2+2}\in E_2\cap
E_1^\prime$.

If $r_1-r_0\geq 2$ and $r_2-r_1\geq 1$, then a change couple
occurs at $(u_{2r_0-1},\ u_{2r_1-1})$, a contradiction. Hence,
$r_1-r_0=1$ or $r_2=r_1$.

If $r_1-r_0=1$, then $r_2> r_1$, and the edge
$u_{2r_2+1}u_{2r_2+2}$ exits. If $r_2-r_1\geq 2$, a change couple
occurs at $(u_{2r_1-1},\ u_{2r_2-1})$, a contradiction. Therefore
$r_2=r_1+1$. Moreover, if any B-change occurs at $u_{2j-1}$ where
$j\geq r_2+1$, we obtain a change couple $(u_{2r_0-1},\
u_{2j-1})$, again leading to a contradiction. Hence, we must have
$u_{2r_2+1}u_{2r_2+2},\ \ldots,\ u_{2r-3}u_{2r-2}\in E_2\cap
E_1^\prime$, and then $r_0=(r-3)/2$, $r_1=(r-1)/2$ and
$r_2=(r+1)/2$.



Then $u_{r+1}u_{r+2}$ and $v_1v_0$ is a critical edge pair, with
the central path $u_{r+3}Ru_{2r-1}v_0^\prime v_1^\prime
u_0Ru_{r}$. Again we may assume that the edge of the central path
not in $M$ are of type II or III. Consider the edges
$u_{r-4}u_{r-3}$ and $u_{r-2}u_{r-1}$, Since $u_{r-4}v_0+$ and
$u_{r-1}v_1+$, we must have $u_{r-3}u_{r+2}+$ and
$u_{r-2}u_{r+1}+$. Since $r\geq 7$, $2r-3> r+3$. Consider the
edges $u_{2r-3}u_{2r-2}$. Since $v_1u_{2r-2}+$, we must have
$u_{2r-3}u_{r+1}+$. But then we find a change couple occur at
$(u_{2r-3},\ u_{r-4})$, a contradiction (Figure
\ref{figure:N11Nredraw}).

\begin{figure}[!htbp]
\centering
\includegraphics[width=0.62\linewidth]{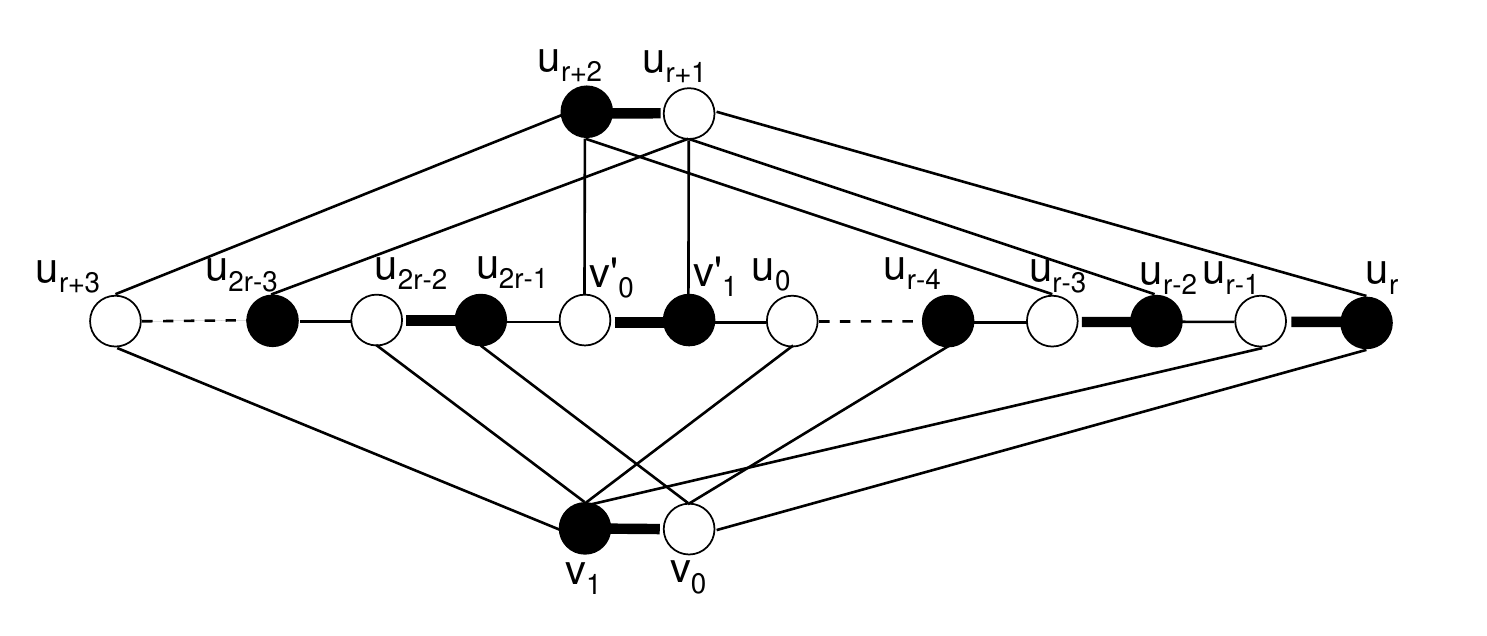}
\caption{Critical pair $u_{r+1}u_{r+2}$ and $v_1v_0$}
\label{figure:N11Nredraw}
\end{figure}

If $r_2=r_1$, then $u_1u_2,\ \ldots,\ u_{r-2}u_{r-1}\in E_1\cap
E_2^\prime$ and $u_{r}u_{r+1},\ \ldots,\ u_{2r-3}u_{2r-2}\in
E_2\cap E_1^\prime$. Then, $u_{r-1}u_{r}$ and $v_0 v_1$ is a
critical pair, with the central path $u_{r+1}Ru_{2r-1}v_0^\prime
v_1^\prime u_0Ru_{r-2}$. For the edges $u_{2i-1}u_{2i}$ with
$(r+3)/2\leq i\leq r-1$, $v_1u_{2i}+$, so we must have
$u_{2i-1}u_{r-1}+$. For the edges $u_{2i-1}u_{2i}$ with $1\leq
i\leq (r-3)/2$, $v_0u_{2i-1}+$, so we must have $u_{2i}u_{r}+$.
For the edge $u_{2r-1}v_0^\prime$ and $v_1^\prime u_0$, we have
$u_{2r-1}v_0+$, $v_0^\prime u_{r}+$, $v_1^\prime u_{r-1}+$ and
$u_0v_1+$. Thus we reach a same config with the case that
$r_1-r_0=1$.

\section{Final Remarks}
Most of the degree sum conditions for Hamilton problems care about
independent vertex sets. In our work, we try to strengthen the
condition of our main theorem, by replacing \textquotedblleft for
every vertex pair $u$ and $v$, where there is not arc from $u$ to
$v$\textquotedblright with \textquotedblleft for every vertex pair
$u$ and $v$\textquotedblright. Naturally, if the former condition
guarantees hamiltonicity without exception, then such a strengthening
brings nothing. But in the case where there are exceptions, we do
find some differences. Let $\mathcal{D}_1^\prime$ be a subset of
$\mathcal{D}_1$, in which $n=m$. Let $\mathcal{D}_3^\prime$ be a
subset of $\mathcal{D}_3$, where $n=1$. We have the following
result.

\begin{theorem} \label{MainTheoremDigraphStrengthen}
Let $D$ be a digraph. If for every vertex pair $u$ and $v$,  we have
$d^+(u)+d^-(v)\geq |D|-1$, then $D$ has a directed Hamilton cycle,
unless $D\in \mathcal{D}_1^\prime$, $\mathcal{D}_2$ or
$\mathcal{D}_3^\prime$, or $D=D_4$.
\end{theorem}

As a corollary, we can improve the Ore condition as well. Given a
(undirected) graph $G$, if we replace every edge $uv\in E(G)$ with
two arcs $uv$ and $vu$, we have a digraph $D$. Applying Theorem
\ref{MainTheoremDigraph} on $D$, we obtain the following result.

Let $n,m\geq 1$, and $\mathcal{G}_5$ be the set of graphs obtained
by identify one vertex
  of a complete graph $K_{m+1}$ and one vertex of a complete graph
  $K_{n+1}$, where $n,\ m\geq 1$. Let $\mathcal{G}_6$ be the set
  of all graphs obtained by joining every vertex of a graph
  $I_{n+1}$ to every vertex of an arbitrary graph on $n$
  vertices.

\begin{corollary}
Let $G$ be a graph. If for every distinct nonadjacent vertex pair $u$
and $v$, we have $d(u)+d(v)\geq |G|-1$, then $G$ has a Hamilton
cycle, unless $G\in \mathcal{G}_5$, or $G\in \mathcal{G}_6$.
\end{corollary}
A slightly stronger result can be found in \cite{LLF2007}. There
is only one exceptional class, for it considers only $2$-connected
graphs.

  \begin{theorem} (Li, Li and Feng, 2007)
    Let G be a $2$-connected graph with $|G|\geq 3$. If $d(u) + d(v)\geq |G|-1$ for
    every pair of vertices $u$ and $v$ with $d(u, v) = 2$, then $G$ has
    a Hamilton cycle, unless $|G|$ is odd and $G\in \mathcal{G}_6$.
  \end{theorem}

  Stimulated by above results, we conjecture that the lower bound of degree sum in the following result can be reduced by
  $1$, with some exceptional cases.

  \begin{theorem} (Bang-Jensen, Gutin and Li, 1996 \cite{BGL1996})
  Let $D$ be a strong digraph such that for every pair of dominating
  non-adjacent and every pair of dominated non-adjacent vertices
  $\{u, v\}$, we have $min\{d^+(u)+d^-(v), d^-(u)+d^+(v)\}\geq
  |D|$. Then $D$ has a directed Hamilton cycle.
  \end{theorem}

\section*{Acknowledgments}

We thank the anonymous reviewers for their careful reading of the paper, and valuable suggestions that have helped to improve the quality of the paper greatly. In particular, the proof of Claim \ref{claim:G_1 complete} and Claim \ref{claim:G_1 G_2 disjoint} is significantly shortened and improved according to the suggestions from one of the reviewers. We also thank Professor Gregory Gutin for sending us a hard copy of reference \cite{Darbinyan1986}.

\end{document}